\documentclass[12pt]{article}
\usepackage[margin=2in,includefoot,footskip=12pt,]{geometry}
\usepackage{amsfonts,amsmath, amsthm, amssymb,hyperref,fullpage}
\usepackage{verbatim}
\usepackage{graphicx}
\usepackage{listings}
\usepackage{xcolor,enumitem}
\usepackage{tikz}
\usepackage{array}
\usepackage{tabularx}
\usepackage{float}
\usetikzlibrary{arrows}
\usetikzlibrary{graphs,graphs.standard}
\usetikzlibrary{shapes.geometric}
\usetikzlibrary{svg.path}
\usepackage{bm}
\usetikzlibrary{patterns}

\newtheorem{theorem}{Theorem}[section]
\newtheorem{lemma}{Lemma}[section]

\newtheorem{remark}{Remark}[section]

\newtheorem{conj}{Conjecture}[section]
\newtheorem{obs} {Observation} [section]

\newcommand{\affl}[3]{%
	\noindent #1,
	\textsc{#3}\\
	Email: \texttt{#2}\\[1.5pt]
}

\title{Maximizing the algebraic connectivity of graphs of given order and size: a proof of a conjecture of Kolokolnikov}

\author{Sebastian M.~Cioab\u{a}, Abhay Jayarajan, \\ M. Rajesh Kannan, Rahul Roy}

\begin{document}
	
	\maketitle
	\begin{abstract}
		The algebraic connectivity of a graph $G$ is a well-studied graph invariant that is related to other properties of the graph such as connectivity and expansion. Given $n$ and $m$, $\alpha(n,m)$ is the maximum algebraic connectivity of a graph with $n$ vertices and $m$ edges. In 2015, Kolokolnikov conjectured that $\alpha(n,2n-4)=2$ for $n\geq 4$, and verified this claim computationally for $n \le 12$. In this paper, we prove Kolokolnikov's conjecture. We also show that  $\alpha(n,3(n-3)) = 3$ is false in general. 
	\end{abstract}
	
	
	
	\section{Introduction}
	
	Throughout this paper, all graphs are finite, simple, and undirected. Let $G=(V(G),E(G))$ be a graph of order $n = |V(G)|$ and size $m = |E(G)|$. The degree of a vertex $v \in V(G)$ is denoted by $d_G(v)$, or simply $d(v)$ when the underlying graph is clear from the context. The adjacency matrix of $G$ is denoted by $A = A(G)$, and its diagonal degree matrix $D=D(G)$ is the diagonal matrix whose entries are the degrees of the corresponding vertices of $G$. The Laplacian matrix of $G$, denoted by $L(G)$ or simply $L$, is defined as $L = L(G) = D - A$. It is known that $L$ is a real symmetric positive semidefinite matrix and that the multiplicity of the eigenvalue $0$ equals the number of connected components of $G$, see \cite[Prop.1.3.7]{BrouwerHaemers} for example. Therefore,  the eigenvalues of $L$ can be arranged as $0 = a_1(G) \le a_2(G) \le \cdots \le a_n(G)$ .
	
	In 1970s, Fiedler \cite{fiedler73} proved that for non-complete graphs $a_2(G)$ is a lower bound for the vertex-connectivity of $G$ and called $a_2(G)$ the \emph{algebraic connectivity} of $G$. Since then, this parameter has been well studied in various theoretical and practical contexts, see \cite{fiedler73,MF75,BM92,SpielmanTeng} for example. For the rest of the paper, we will use the notation $a(G) = a_2(G)$.
	
	A natural problem is to maximize the
	algebraic connectivity over all graphs having a prescribed order and size.
	For positive integers \(n\) and \(m\), let
	\[
	\mathcal{G}_{n,m}
	=
	\{G:|V(G)|=n,\ |E(G)|=m\},
	\]
	and define
	\[
	\alpha(n,m)
	=
	\max\{a(G):G\in\mathcal{G}_{n,m}\}.
	\]
	A graph attaining this maximum is called an
	\emph{algebraic-connectivity maximizing graph} in
	\(\mathcal{G}_{n,m}\). Although algebraic connectivity is non-decreasing
	under the addition of edges, determining \(\alpha(n,m)\) is highly
	nontrivial because the placement of the prescribed number of edges has a
	substantial effect on the Laplacian spectrum, see also \cite{mosk2008}. When restricted to regular graphs of given order and size, this problem was studied by Nozaki \cite{NozakiLP}, see also \cite{CKNV}.

	It is known that $\alpha(n,n-1)=1$ and that the unique extremal graph is the star $K_{1,n-1}$ with $n$ vertices, see \cite{maas1987transportation, merris1987characteristic}.
	Lal, Patra and Sahoo \cite{lal2011algebraic} proved that $\alpha(n,n)$ is attained by the cycle $C_n$ for $n \in \{3,4,5\}$ and by $S_n^+$ for $n \geq 7$ uniquely, where $S_n^+$ is the graph obtained from star graph $K_{1,n-1}$ by adding one edge. For $n=6$,  $C_6$ and $S_6^+$ are two graphs that maximize the algebraic connectivity among all graphs on six vertices and six edges.
	
	In 2015, Kolokolnikov \cite{Kolo15} proposed the following conjecture.
	\begin{conj}\label{conj:Kolo}
		Let \(n\ge4\), and let \(G\) be a graph on \(n\) vertices and
		\(2(n-2)\) edges. Then
		\begin{equation}\label{eq:conjK}
			a(G)\le a(K_{2,n-2})=2.
		\end{equation}
	\end{conj}
	Equivalently, Conjecture~\ref{conj:Kolo} asserts that
	$\alpha(n,2n-4)=2$. In \cite{Kolo15}, the author verified this conjecture computationally using the package \textsc{nauty} for $n\leq 12$. These computations also identified, for each such $n$, the graphs attaining the largest algebraic connectivity.
	
	In this paper, we prove  Conjecture~\ref{conj:Kolo} for all $n \ge 12$ and the main result of our paper is the following theorem.
	\begin{theorem}\label{thm}
		If $G$ is a graph with $n \ge 12$ vertices and $2n-4$ edges, then
		\[
		a(G) \le 2.
		\]
		Consequently, $\alpha(n,2n-4) = 2$.
	\end{theorem}
	Independently of our work and around the same time, a proof of Conjecture~\ref{conj:Kolo} for $n\geq 123$ has been posted in  \cite{zhu2026maximizingalgebraicconnectivity2n2} and another proof has appeared in  \cite{chi2026}.

	Our arguments are entirely structural for \(n\ge12\) and do not depend on computer enumeration. In  addition to settling Conjecture~\ref{conj:Kolo}, our tools
	show how variational methods, degree-deficit counting, small edge boundary density and inertia techniques can be combined in algebraic-connectivity maximization problems
	with a prescribed number of vertices and edges and may be useful for other related problems.
	
	This paper is organized as follows. In Section \ref{nota-tools}, we introduce the notation and collect the relevant known results. In Section \ref{outline}, we provide an outline of the proof of the conjecture. In Section \ref{observation-sec}, we establish several observations that will be useful in the proof. In Section \ref{proof-main}, we present the proof of the conjecture. Finally, in Section \ref{final-rem}, we conclude with some remarks. 
	
	\section{Notations and Tools}\label{nota-tools}
	
	Let $G$ be a graph and $u,v$ be two distinct vertices of $G$. We write $u \sim v$ if $u$ and $v$ are adjacent. For a subset $P$ of $V(G)$ denote by $N(P)$ the set of vertices that are adjacent to at least one vertex in $P$. For a vertex $u \in V(G)$, by $N_P(u)$ we mean the set of vertices in $P$ that are adjacent to $u$ and we write $d_P(u) =\vert N_P(u) \vert$. We use $G[P]$ to denote the subgraph of $G$ induced by the vertices in $P$. For subsets $P, Q$ of $V(G)$, we denote by $e(P, Q)$ the number of edges in $G$ with one endpoint in $P$ and the other in $ Q$. We denote $ e(P, P^c)$ by $\vert \partial P \vert $. Also, $e(P)$ denotes the number of edges in $G[P]$, that is, the number of edges with both endpoints in $P$. We use the notations $e(P)$ and $e(G[P])$ interchangeably. The set $P$ is called an independent set if no two vertices in $P$ are adjacent.
	
	The algebraic connectivity admits the well-known variational characterization
	\[
	a(G)=
	\min_{\substack{x\in\mathbb{R}^{n}\setminus\{0\}\\
			x\perp\mathbf{1}}}
	\frac{x^{T}L(G)x}{x^{T}x}
	=
	\min_{\substack{x\in\mathbb{R}^{n}\setminus\{0\}\\
			x\perp\mathbf{1}}}
	\frac{\displaystyle\sum_{uv\in E(G)}(x_u-x_v)^2}
	{\displaystyle\sum_{v\in V(G)}x_v^2},
	\]
	where $\mathbf{1}$ denotes the all-ones vector of appropriate size and $x \perp \mathbf{1}$ means $\sum \limits_{u \in V} x_u=0$, see \cite{zbMATH00867649} for example. 
	Equivalently, 
	\begin{theorem}[\cite{MF75} ]\label{var_char_2}
		Let $G$ be a graph of order $n$. Then
		\[
		a(G)
		=
		\min
		\frac{
			n \displaystyle\sum_{ij\in E}(x_i-x_j)^2
		}{
			\displaystyle\sum_{i,j\in V,\; i<j}(x_i-x_j)^2
		},
		\]
		where the minimum is taken over all non-constant vectors
		$x\in\mathbb{R}^n$.
	\end{theorem}


	This all-pairs form of the variational characterization is particularly
	convenient for constructing trial vectors that are constant on suitably
	chosen vertex subsets. These trial vectors will play an important role in our
	proof. The following results will be useful in the proofs presented in the subsequent sections.
	
	\begin{theorem}[{\cite{ZL22}, Theorem 1.1}]\label{part thm}
		Let $G$ be a connected graph with $n$ vertices, and let $V(G) = V_1 \cup V_2 \cup V_3$ be a partition of $G$ with $n_i = |V_i| \le n-2$ for $i = 1, 2, 3$. Then
		\begin{equation}
			a(G) \le \min \left\{ \frac{\beta - \sqrt{\beta^2 - 4\alpha \gamma}}{2\alpha} \right\}, \
		\end{equation}
		where $\alpha = n_1 n_2 n_3$, $\beta = (t_1 + t_2)n_2 n_3 + (t_2 + t_3)n_1 n_2 + (t_1 + t_3)n_1 n_3$, $\gamma = n(t_1 t_2 + t_2 t_3 + t_1 t_3)$, and $t_1 = e(V_1, V_2)$, $t_2 = e(V_1, V_3)$ and $t_3 = e(V_2, V_3)$ and minimum is taken over all admissible partitions.
	\end{theorem}
	\begin{remark}\label{poly_trick}
		For a fixed partition, the right hand side in the upper bound is the smallest root of the polynomial $p(x)=\alpha x^2-\beta x+\gamma $. Suppose for $y \in \mathbb{R}$, we have that $p(y) \leq 0$. Since $\alpha > 0$, $p$ is an upward opening parabola, which forces $y$ to be at least the smallest root of $p(x)$. Thus, to prove $a(G)\leq y$, it suffices to show that $p(y) \leq 0$. In addition if $p(y) < 0$ then $a(G) < y$.  
	\end{remark}
	The following is one of the key tools used in the proof of the main result. 
	\begin{theorem}[{\cite{Liu}, Lemma 3.2}]\label{Liu}
		Let $G=(V,E)$ be a connected graph. If $X$ and $Y$ are non-empty vertex disjoint subsets of $V$ with $a(G) \geq \max \bigg\{\frac{|\partial X|}{|X|}, \frac{|\partial Y|}{|Y|}\bigg \}$, then 
		\begin{equation*}
			e(X,Y)^2\geq |X||Y|\left(a(G)-\frac{|\partial X|}{|X|}\right)\left(a(G)-\frac{|\partial Y|}{|Y|}\right).
		\end{equation*}
	\end{theorem}
	
	The following is an immediate consequence of the above theorem: if $X$ and $Y$ non-empty disjoint subsets of $V(G)$ such that $a(G)>\frac{|\partial X|}{|X|}$ and $a(G)>\frac{|\partial Y|}{|Y|}$, then there exists at least one edge between the sets $X$ and $Y$.
	
	\section{Outline of the proof}\label{outline}
	
	Our proof of the conjecture involves several steps, and we briefly outline the main ideas here. Since a disconnected graph
	has algebraic connectivity zero, it is enough to consider connected graphs. The minimum degree inequality $\delta(G) \leq \dfrac{4n-8}{n}$ forces $\delta(G) \leq 3$. If $\delta(G)\leq 2$, then as $a(G)$ is at most the minimum degree of $G$, we get that $a(G)\leq 2$. Thus we restrict our work to graphs with $\delta(G) =3$.
	Define
	\[
	T=\{v\in V(G):d(v)=3\}
	.
	\]
	Our proof strategy is summarized in Table \ref{tab:case-division}.
	
	\begin{table}[H]
		\centering
		\small
		\renewcommand{\arraystretch}{1.6}
		\setlength{\tabcolsep}{6pt}

		\begin{tabularx}{\textwidth}{
				|>{\centering\arraybackslash}p{1.6 cm}
				|>{\raggedright\arraybackslash}X
				|>{\centering\arraybackslash}p{3.8cm}|
			}
			\hline
			\textbf{Sr.\ No.}
			&
			\multicolumn{1}{c|}{\textbf{Cases}}
			&
			\textbf{Proof}
			\\
			\hline
			
			1
			&
			$\delta(G)\leq 2$;  $\delta(G)\geq 4$ are not possible
			&
			Observation \ref{min-degree-2}
			\\
			\hline
			
			2
			&
			$\delta(G)=3$ and $T$ is not an independent set, where $n \geq 12$
			&
			Theorem \ref{not-independent}
			\\
			\hline
			
			3
			&
			$\delta(G)=3$ and $T$ is an independent set
			&
			Cases \(3(a)\)--\(3(d)\)
			\\
			\hline
			
			\(3(a)\)
			&
			$G$ has no vertex of degree \(4\), where $n \geq 14$
			&
			Theorem \ref{deg5_vertices}
			\\
			\hline
			
			\(3(b)\)
			&
			$G$ has a vertex of degree \(4\), where
			\(14\leq n\leq19\)
			&
			Lemma \ref{small vertices}
			\\
			\hline
			
			\(3(c)\)
			&
			$G$ has a degree-\(4\) vertex \(u\in T^{c}\) such that
			\[
			d_T(u)\geq2,
			\qquad n\geq20
			\]
			&
			Lemma \ref{large-T-degree}
			\\
			\hline
			
			\(3(d)\)
			&
			Every degree $4$ vertex \(u\in T^{c}\) satisfies
			\[
			d_T(u)\leq1,
			\qquad n\geq32
			\]
			&
			Lemma \ref{adjacent-vertex-in-R} and \ref{large-R-degree}, and Theorem\ref{final-theorem}
			\\
			\hline

		\end{tabularx}
		
		\vspace{2mm}
		\caption{The sketch of the proof of Conjecture~\ref{conj:Kolo}.}
		\label{tab:case-division}
	\end{table}
	
	\section{Observations}\label{observation-sec}
	
	In this section, we collect several crucial observations that lead to the proof of the conjecture. Since the conjecture holds trivially for disconnected graphs, we assume throughout this manuscript that all graphs under consideration are connected. All throughout this section, $G$ is a graph with $n$ vertices and $2n-4$ edges.
	
	\begin{obs}\label{min-degree-2}
		Let $\delta(G)$ denote the minimum degree in the graph $G$. If $\delta(G)\leq 2$, then  $a(G)\leq \delta(G)\leq 2,$ and we are done in this case. Note that $\delta(G)\leq \frac{2(2(n-2))}{n}=4-\frac{8}{n}$. Therefore, $\delta(G)<4$ and consequently, $\delta(G)\leq 3$. Hence, we only need to check the case that $\delta(G)=3.$ 
	\end{obs}
	
	Thus we now consider only the graphs $G$ with $\delta(G)=3$.
	\begin{obs}\label{t_geq_8}
		Define
		\begin{equation}
			T=\biggl\{u \in V(G) : d(u) =3\biggl\}.
		\end{equation}
		and $t = \vert T \vert$. Note that 
		\begin{equation}\label{eq:T}
			4n-8=2m = \sum\limits_{u \in V(G)} d(u) \geq 3t + 4(n-t).
		\end{equation}
		Hence, $t\geq 8$.
	\end{obs}

	The excess degree of a vertex $u$ in $T^c$ is $d_G(u)-4$. The following observation will also be used in our proof.
	\begin{obs}\label{excess degree}
		$$\sum\limits_{u \in T^c}(d(u)-4)=t-8.$$  
	\end{obs}
	\begin{proof}
		Note that $$4n-8=2m=\sum\limits_{u\in T}d(u)+\sum\limits_{v\in T^c}d(v)=3t+\sum_{v\in T^c}d(v)$$
		and therefore,
		\begin{equation*}
			\sum_{v\in T^c}d(v)=4n-8-3t=4(n-t)+t-8,
		\end{equation*}
		and 
		\begin{equation*}
			\sum_{v\in T^c}(d(v)-4)=t-8.
		\end{equation*}
		
	\end{proof}
	
	\section{Proof of the Conjecture \ref{conj:Kolo}} \label{proof-main}
	In this section, we prove Conjecture~\ref{conj:Kolo}. We divide the proof into two cases according to the structure of the subgraph induced by the set
	\[
	T=\{v\in V(G): d_G(v)=3\}.
	\]
	In Subsection~\ref{not_indep}, we prove the conjecture when $T$ is not an independent set. This case admits a comparatively short proof using Theorem \ref{Liu}. In Subsection~\ref{indep}, we consider the remaining case in which $T$ is an independent set. This case requires a more detailed structural analysis together with variational and Schur complement techniques.
	\subsection{The set $T$ is not independent}\label{not_indep}
	\begin{theorem}\label{not-independent}
		Let $G$ be a graph on $n$ vertices and $2n-4$ edges, where $n \geq 12$. If $T$ is not an independent set, then $a(G) \leq 2.$
	\end{theorem}
	\begin{proof}
		Let $x_1x_2$ be an edge in  $G[T]$. Define $X=\{x_1,x_2\}$. Since $d(x_1)=d(x_2)=3$, we have $\vert \partial X \vert =4$ and $\frac{\vert \partial X \vert}{\vert X \vert}=2$.
		Let $Q=N(X) \setminus X$. Again by the degree argument, it is clear that $|Q| \leq 4$. 
		
		Consider the set $Y=V(G) \setminus (X \cup Q)$. Then $|Y|=n-\vert Q \vert-2$.\\
		Now, 
		$\vert \partial Y \vert=e(Q,Y)+e(Y,X)$. Since $e(Y,X)=0$, we have $\vert \partial Y \vert=e(Q,Y)$. Hence, 
		\begin{align*}
			e(Q,Y)&=   \sum\limits_{u \in Q} d(u)-2e(Q)-e(Q,X)\\
			&=\sum\limits_{u \in Q \cap T} d(u)+\sum\limits_{u \in Q \cap T^c} d(u)-2e(Q)-4\\
			&=3\vert Q \cap T \vert+\sum\limits_{u \in Q \cap T^c} (d(u)-4)+4(\vert Q \vert - \vert Q \cap T \vert)-2e(Q)-4\\
			\text{So,} \quad \quad \quad  \vert \partial Y \vert =e(Q,Y) &= 4\vert Q \vert - \vert Q \cap T \vert +\sum\limits_{u \in  Q \cap T^c} (d(u)-4)-2e(Q)-4.\\
		\end{align*}
		Therefore,
		\begin{align*}\label{Part/Card}
			2\vert Y \vert- \vert \partial Y\vert&= 2n-2\vert Q \vert -4 -4\vert Q \vert + \vert Q \cap T \vert -\sum\limits_{u \in  Q \cap T^c} (d(u)-4)+2e(Q)+4  \notag \\
			&\geq 2n-6\vert Q \vert + \vert Q \cap T \vert -\sum\limits_{u \in   T^c} (d(u)-4)+2e(Q) \notag \\
			&=2n-6\vert Q \vert + \vert Q \cap T \vert -t + 8 +2e(Q) \quad \quad \mbox{(Using Observation \ref{excess degree})} \notag \\
			&=(2n-5|Q|+8)-(|Q|+|T|-|Q\cap T|)+2e(Q) \notag \\
			&=(2n-5|Q|+8)-|Q\cup T|+2e(Q) \\
			&=(n-5|Q|+8)+(n-|Q\cup T|)+2e(Q)\\
			&\geq (n-12)+0+0, \text{ as } |Q|\leq 4\\
			&\geq 0, \text{ as } n\geq 12.
		\end{align*} 
		Hence, $2|Y|\geq |\partial Y|$.
		
		
		
		
		By contradiction, assume that $a(G) > 2$. As $\frac{\vert \partial Y \vert}{\vert Y \vert}\leq 2$ and $\frac{\vert \partial X \vert}{\vert X \vert}= 2$, by using Theorem \ref{Liu} for the sets $X$ and $Y$, we deduce that $e(X,Y) >0 $. This is a contradiction with $e(X,Y)=0$. Hence, $a(G) \leq 2$.
	\end{proof}
	
	\subsection{The set $T$ is independent}\label{indep}
	
	We now focus on the case when $T$ is independent. Since $T$ is independent and every vertex of $T$ has degree $3$, $e(T,T^c)=3t$. Therefore, $$ e(G[T^c]) =\vert E(G) \vert-e(T,T^c),$$ and hence
	\begin{equation}\label{size of T}
		e(G[T^c])= 2n-4-3t.  
	\end{equation}
	Since $e(G[T^c])\geq 0$ and $t \geq 8$, we get that $n \geq 14$.
	
	Next, we consider two cases: (i) the graph \(G\) has no vertex of degree \(4\), and (ii) the graph \(G\) contains at least one vertex of degree \(4\).
	\subsubsection{Case (i) : The graph $G$ has no degree $4$ vertex}
	In the following observations, we assume that $T$ is an independent set and every vertex in $T^c$ has degree at least $5$. We prove the result below as a consequence of a series of observations.
	\begin{theorem}\label{deg5_vertices}
		Let $n \geq 14$. If $T$ forms an independent set and $G$ does not contain a degree $4$ vertex, then $a(G) \leq 2$.
	\end{theorem}
	\begin{obs}\label{LB of T}
		$t \geq 20$.
		\begin{proof}
			Because $2n=3t+4+e(G[T^c])$, we deduce that $2n\geq 3t+4$. From Observation \ref{excess degree} , we have that 
			\begin{align}\label{upper-bound-T-comp}
				t-8=\sum\limits_{u \in T^c}(d(u)-4)\geq |T^c|=n-t.
			\end{align}
			Hence, 
			\begin{equation}\label{n_leq_2t_minus_8}
				n\leq 2t-8.
			\end{equation} Therefore, $4t-16\geq 2n\geq 3t+4$ which implies that $t\geq 20$.
		\end{proof}
	\end{obs}
	
	Define
	$$
	S=\biggl\{u \in T^c | d_{T^c}(u)\leq 2 \biggr\}.
	$$
	Let $s=|S|$.
	\begin{obs}\label{LB of s}
		$s \geq \frac{t+16}{3}$.
		\begin{proof}
			For each vertex $u\in T^c \setminus S$, $d_{T^c}(u)\geq 3$. Therefore,
			\begin{equation*}
				4n-8-6t=2e(T^c)=\sum_{v\in T^c}d_{T^c}(v)\geq \sum_{u\in T^c\setminus S}d_{T^c}(u)\geq 3|T^c\setminus S|=3(n-t-s).
			\end{equation*}
			Hence, $3s\geq 3t+8-n$. By equation (\ref{n_leq_2t_minus_8}), $n\leq 2t-8$. Thus $3s\geq t+16$ and $s\geq \frac{t+16}{3}$.
		\end{proof}
	\end{obs}
	
	\begin{obs}\label{edges in S}
		\begin{equation*}
			e(S) \leq \frac{\vert T^c \vert -2 t -8+3s }{2}=\frac{n+3s-3t-8}{2}.
		\end{equation*}
		\begin{proof} We have that 
			\begin{equation*}
				2e(S)=\sum\limits_{u \in S}d_{S}(u)\leq \sum\limits_{u \in S}d_{T^c}(u).
			\end{equation*}
			For each vertex $u\in T^c \setminus S$, $d_{T^c}(u)\geq 3 $. Hence,  
			\begin{align*}
				\sum\limits_{u \in S}d_{T^c}(u)&=\sum\limits_{u \in T^c}d_{T^c}(u)-\sum\limits_{u \in T^c \setminus S}d_{T^c}(u)=2e(T^c)-\sum\limits_{u \in T^c \setminus S}d_{T^c}(u)\\
				&\leq 2e(T^c)-3(\vert T^c \vert -s)=2(2n-4-3t)-3(n-t-s)\\
				&=n+3s-3t-8.
			\end{align*}
			Hence, the result follows.
		\end{proof}
	\end{obs}
	
	For $u \in S$, define
	$X_u=\{u\} \cup N_{T}(u)$.
	\begin{obs}\label{edge expansion}
		$ \frac{\vert \partial X_u \vert}{\vert  X_u \vert} \leq 2$.
	\end{obs}
	\begin{proof}
		We have that $\vert X_u\vert=d_{T}(u)+1$. Because $T$ is independent, each vertex in $N_T(u)$ has exactly two neighbors outside $X_u$. Also, $u$ has exactly $d_{T^c}(u)$ neighbors outside $X_u$. Because $u\in S$, $d_{T^c}(u)\leq 2$. Therefore,
		\begin{equation*}
			\vert \partial X_u \vert=2d_{T}(u)+d_{T^c}(u)\leq 2d_T(u)+2=2|X_u|.
		\end{equation*}
	\end{proof}

	Define
	$$
	\mathcal{P}=\biggl\{\{u_1,u_2\} \subseteq S : N_T(u_1) \cap N_T(u_2)\neq \phi \biggr\}.
	$$
	\begin{obs}\label{non-adjacent pairs are in M}
		If $a(G) > 2$, $\{u_1,u_2\} \subseteq S$ and $u_1 \nsim u_2$ then $\{u_1,u_2\} \in \mathcal{P}$.
	\end{obs}
	\begin{proof}
		Assume that $a(G)>2$. By contradiction, suppose that there is $\{u_1,u_2\}\subseteq S$ such that   $u_1 \nsim u_2$ and $\{u_1,u_2\} \notin \mathcal{P}$. Hence, $N_T(u_1) \cap N_T(u_2)= \phi$. Because $u_1 \nsim u_2$ and $T$ is independent, $e(X_{u_1},X_{u_2})=0$. 
		
		By Observation \ref{edge expansion}, we have that $ \frac{\vert \partial X_{u_1} \vert}{\vert  X_{u_1} \vert} \leq 2$ and $ \frac{\vert \partial X_{u_2} \vert}{\vert  X_{u_2} \vert} \leq 2$. Using Theorem \ref{Liu} for the sets $X_{u_1}$ and $X_{u_2}$, we obtain that $e(X_{u_1},X_{u_2})>0$, which is a contradiction.
	\end{proof}

	\textbf{Proof of Theorem \ref{deg5_vertices}:}        By contradiction, assume that $a(G)> 2$. By Observation \ref{non-adjacent pairs are in M}, we have that
	$$
	\binom{s}{2}=\binom{\vert S \vert}{2}=e(G[S])+e(G^c[S])\leq e(G[S])+\vert \mathcal{P} \vert.
	$$
	A vertex $u\in T$ has degree $3$, so it has at most $3$ neighbors in $S$. It can therefore produce at most $\binom{3}{2}=3$ pairs of vertices in $S$ having $u$ as a common neighbor. Summing over $T$, the number of distinct pairs sharing a $T$-neighbor is at most $3t$. Thus, $|\mathcal{P}|\leq 3|T|$ and 
	$$
	\binom{s}{2}\leq e(G[S])+3 t.
	$$
	By using Observation \ref{edges in S}, we get
	\begin{align*}
		\frac{s(s-1)}{2} &\leq \frac{\vert T^c \vert -2 t-8+3s}{2}+3t\\
		s^2 -s &\leq \vert T^c \vert -2 t-8+3s+6t\\
		s^2-4s&\leq \vert T^c \vert +4 t-8.
	\end{align*}
	From equation \eqref{upper-bound-T-comp}, we get
	\begin{align}
		s^2 -4s \leq 5 t -16.\label{UB or r^2-4r}
	\end{align}
	Also, by Observation \ref{LB of T} and \ref{LB of s}, we obtain
	$$
	s \geq \frac{t +16}{3}\geq 12.
	$$
	Note that the function $x^2-4x$ is increasing for $x\geq 2$. Therefore
	\begin{align}
		s^2 -4s &\geq \Bigg(\frac{t +16}{3}\Bigg)^2-4 \Bigg(\frac{t +16}{3}\Bigg)\notag\\
		&=\frac{t ^2+32t+256-12 t-192}{9}\notag\\
		&=\frac{t^2+20t+64}{9}.\notag
	\end{align}
	But 
	$$
	\frac{t^2+20t+64}{9}-(5t-  16)=\frac{t^2+20t+64-45t+144}{9}=\frac{t ^2-25t +208}{9}.
	$$
	The polynomial 
	$$
	x^2-25x +208=(x -\frac{25}{2})^2+\frac{207}{4}
	$$
	is strictly positive for all real $x$. Hence
	$$
	s^2-4s >5 t -16,
	$$
	contradicting (\ref{UB or r^2-4r}). Therefore the assumption $a(G) >2$ is impossible, proving
	$$
	a(G)\leq 2.
	$$
	\qed
	\subsubsection{Case (ii): The graph $G$ has a degree $4$ vertex}
	In the following, we assume that $T^c$ has at least one vertex of degree $4$. 
	
	Recall that $n \geq 14$. In the following observations, we assume that $T$ is an independent set and there exists at least one vertex in $T^c$ having degree $4$.
	Define $H:= \{ u \in T^c: d(u) \geq 5\}$.  We first prove the conjecture for small numbers of vertices.
	\begin{lemma}\label{small vertices}
		Let $14 \leq n \leq 19$. If $T$ forms an independent set and $G$ has at least one vertex of degree $4$, then $a(G) \leq 2$.
	\end{lemma}
	\begin{proof}
		We prove that there is at least one vertex of degree $4$ in $T^c$ that has at least $3$ neighbors in $T$. Suppose on the contrary that every vertex of degree $4$ has at most $2$ neighbors in $T$. Since $T$ is independent, we can count the number of edges from $T$ to $T^c$ as follows:
		\begin{align*}
			3t &=\sum \limits_{u \in T^c}d_T(u)\\
			&\leq 2\vert T^c \setminus H \vert + \sum \limits_{u \in H}d_T(u)\\
			&\leq 2\vert T^c \vert - 2\vert H \vert + \sum \limits_{u \in H}d(u)\\
			&= 2\vert T^c \vert + 2\vert H \vert + \sum \limits_{u \in H}(d(u)-4)\\
			&\leq 2\vert T^c \vert + 2\sum \limits_{u \in H}(d(u)-4) + \sum \limits_{u \in H}(d(u)-4)\\
			&= 2\vert T^c \vert +3\sum \limits_{u \in T^c}(d(u)-4)\\
			&= 2\vert T^c \vert +3t -24. \quad \text{(By Observation \ref{excess degree})}
		\end{align*}
		This gives, $\vert T^c \vert \geq 12$. But $\vert T^c \vert =n-t \leq 19-8=11$, which is a contradiction. Thus there is a vertex in $G$ of degree $4$ adjacent to at least $3$ vertices in $T$. Let  $v$ be a vertex in $T^c$ such that $d(v) = 4$ and $d_T(v) \in \{3,4\}$. 
		Define $A= N_T(v), \quad B=T \setminus A, \quad  C= \{v\}, \quad D= T^c \setminus \{v\}$. Set $k=d_T(v)$. 
		We define a vector $x\in \mathbb{R}^n$ as follows: 
		$$
		x_w =
		\begin{cases}
			3, & \text{if } w \in A,\\
			-1, & \text{if } w \in B,\\
			4, &\text{if } w\in C, \\
			0, & \text{if } w\in D. \\
		\end{cases}
		$$
		We use this trial vector in  Theorem \ref{var_char_2} to complete the proof of this case.
		
		Note that $e(A,B)=0,\quad  e(A,C)=k,\quad e(A,D)=2k, \quad e(B,C)=0, \quad e(B,D)=3(t - k), \quad e(C,D) = 4-k$. Therefore 
		\begin{align*}
			\sum \limits_{ab \in E(G)}(x_a-x_b)^2&=  k(3-4)^2+2k (3-0)^2+3(t - k)(-1-0)^2+(4-k)(4-0)^2\\
			&=k+18 k+3t -3 k+64-16k \\
			&=3 t + 64.
		\end{align*}
		On the other hand,
		\begin{align*}
			&\sum \limits_{\{a,b\}\subseteq V(G) }(x_a-x_b)^2\\&=16k(t -k)+k+9k(\vert T^c \vert -1)+25(t - k)+(t -k)(\vert T^c \vert -1)+16(\vert T^c \vert -1)\\
			&=16kt -16k^2+k+9k\vert T^c \vert -9k+25 t -25k+t \vert T^c \vert - t -k \vert T^c \vert +k+16\vert T^c \vert -16\\
			&=8kt -16k^2-32k+8kn+8 t -16+n t +16n-t^2.
		\end{align*}
		\textbf{Case 1:  } Now let us consider the case $k=4$. Consider the quantity,
		\begin{align*}
			2\sum \limits_{\{a,b\}\subseteq V(G) }(x_a-x_b)^2-n \sum \limits_{ab \in E(G)}(x_a-x_b)^2&=
			2n t +96n-2t^2+80 t -800 -n(3t + 64)\\
			&=n (32-t) -2(t^2-40 t +400)\\
			&=n (32-t) -2(t-20)^2. \\
		\end{align*}
		We have $14 \leq n \leq 19$. Also, by equation \eqref{size of T}, $$8 \leq t \leq \bigg \lfloor \dfrac{2n-4}{3} \bigg \rfloor \leq 11.$$
		Therefore,
		$$n(32-t ) \geq 14\times 21=294.$$
		$$2(20-t)^2 \leq 2 \times 12^2=288. $$
		Hence, $$2\sum \limits_{\{a,b\}\subseteq V(G) }(x_a-x_b)^2-n \sum \limits_{ab \in E(G)}(x_a-x_b)^2 \geq 6> 0.$$
		
		Thus, we have $$\dfrac{n \sum \limits_{ab \in E(G)}(x_a-x_b)^2}{\sum \limits_{\{a,b\}\subseteq V(G) }(x_a-x_b)^2} < 2.$$
		By Theorem \ref{var_char_2}, we get that $a(G) < 2$.
		
		Note that, for $n=14$, equation \eqref{size of T} gives $t \leq 8$. By Observation \ref{t_geq_8},  $t \geq8$. That is,  $t = 8$. Thus, by equation \eqref{size of T}, we get $e(G[T^c]) =0$. By Observation \ref{excess degree}, every vertex in $T^c$ must have degree $4$. Thus $d_T(v)=4$. Similarly for $n = 15$, we get $t = 8, e(G[T^c]) =2$ and hence in this case $v$ can be chosen among an isolated vertex in $G[T^c]$ with $d_T(v)=4$.
		
		Thus for cases $n=14$ and $n =15$, there exists a vertex $v$ with $d_T(v)=4$, and hence the proof is done. Hence, for the remaining case it is enough to consider $n \geq 16$.
		
		\textbf{Case 2: } Now let us look at the case $k=3$. Consider the quantity,
		\begin{align*}
			2\sum \limits_{\{a,b\}\subseteq V(G) }(x_a-x_b)^2-n \sum \limits_{ab \in E(G)}(x_a-x_b)^2&=
			2n t +80n-2t^2+64 t -512 -n(3t + 64)\\
			&=16n+64t -512 -nt -2t^2 \\
			&=(16-t)(n+2t -32).
		\end{align*}
		
		We have $16-t > 0 \text{ and } n+2 t -32 \geq 16+16-32 =0$. Thus $$2\sum \limits_{\{a,b\}\subseteq V(G) }(x_a-x_b)^2-n \sum \limits_{ab \in E(G)}(x_a-x_b)^2 \geq 0.$$

		Hence, we have $$\dfrac{n \sum \limits_{ab \in E(G)}(x_a-x_b)^2}{\sum \limits_{\{a,b\}\subseteq V(G) }(x_a-x_b)^2} \leq 2.$$
		By Theorem \ref{var_char_2}, we get that $a(G) \leq 2$.
	\end{proof}
	We now prove that if there is a degree four vertex that is adjacent to at least two vertices of degree three, the conjecture is true.
	\begin{lemma}\label{large-T-degree}
		Let $n \geq 20$ and $T$ be independent. If there exists a degree $4$ vertex $u \in T^c$ such that $d_T(u)\geq 2$, then $a(G) \leq 2$.
	\end{lemma}
	\begin{proof}
		Let $u \in T^c$ with $d_T(u) \geq 2$. Let $u_1, u_2 \in T$ be distinct neighbors of $u$. Define $X =\{u,u_1,u_2\}$. Since $d(u)=4, d(u_1)=d(u_2)=3$ and $G[X]$ has exactly two edges, we get that $\vert \partial X \vert = 10-2(2)=6=2\vert X \vert$. Next define $Q=N(X)\setminus X$ and $Y=V(G) \setminus \{X\cup Q\}$. Again using the degree argument, is clear that $\vert Q \vert \leq 6$. Now consider,\\
		\begin{align*}
			\vert \partial Y \vert &=e(Q,Y) \leq \sum\limits_{v \in Q}d(v) - \vert \partial X\vert \\
			&= \sum\limits_{v \in Q\cap T}d(v) + \sum\limits_{v \in Q \cap T^c}d(v) -6\\
			&=3\vert T \cap Q \vert + \sum\limits_{v \in Q \cap T^c}d(v) -6\\
			&=3\vert T \cap Q \vert + \sum\limits_{v \in Q \cap T^c}(d(v)-4) + 4\vert Q \cap T^c \vert  -6\\
			&\leq 3 \vert T \cap Q \vert + t  -8 + 4(\vert Q \vert - \vert T \cap Q))-6 \quad \quad  \text{(Using Observation \ref{excess degree})}\\ 
			&=t -14-\vert T \cap Q \vert +4 \vert Q \vert.
		\end{align*}
		Since $\vert Q \vert \leq 6$, we get,\\
		\begin{align*}
			2\vert Y \vert - \vert \partial Y \vert &\geq 2(n-3-\vert Q\vert)-t+14+\vert T \cap Q \vert -4\vert Q \vert\\
			&=2n-t-6\vert Q \vert +8 + \vert T \cap Q \vert\\
			&\geq 2n-t-6\vert Q \vert+8\\
			&\geq 2n-(t+28).
		\end{align*}
		Note that, since $T$ is independent, using \eqref{size of T} we have that,  $t \leq \dfrac{2n-4}{3}$. Therefore, for $n \geq 20$,
		\begin{equation*}
			t+28 \leq \dfrac{2n-4}{3}+28 \leq 2n.
		\end{equation*}
		Hence we get that, $2 \vert Y \vert -\vert \partial Y \vert \geq 0$, that is, $\dfrac{\vert \partial Y \
			\vert}{\vert Y \vert }\leq 2$.
		By contradiction, assume that $a(G) > 2$.
		Since $\frac{\vert \partial Y \vert}{\vert Y \vert}\leq 2$ and $\frac{\vert \partial X \vert}{\vert X \vert}= 2$, by using Theorem \ref{Liu} for the sets $X$ and $Y$, we have $e(X,Y)^2 > 0 $,  which is a contradiction. Hence, $a(G) \leq 2$.
	\end{proof}
	Hereafter, we can assume that any degree $4$ vertex in $T^c$ has at most one neighbor in $T$. Recall that $H= \{ u \in T^c: d(u) \geq 5\}$. Using Observation \ref{excess degree}, we get the following
	\begin{align*}
		3t &= \sum \limits_{u \in T^c}d_T(u)\\
		&= \sum \limits_{u \in T^c \setminus H}d_T(u) + \sum \limits_{u \in H}d_T(u)\\
		&\leq \vert T^c \setminus H \vert + \sum \limits_{u \in H}d_T(u)\\
		&\leq \vert T ^c \vert -\vert H \vert + \sum \limits_{u \in H}d(u) \\
		&= \vert T ^c \vert -\vert H \vert + \sum \limits_{u \in H}(d(u)-4) +4 \vert H \vert\\
		&\leq \vert T ^c \vert+ \sum \limits_{u \in H}(d(u)-4)+3\sum \limits_{u \in H}(d(u)-4)\\
		&\leq \vert T ^c \vert+ 4\sum \limits_{u \in T^c}(d(u)-4)\\
		&=\vert T^c \vert +4(t -8)\\
		&= n-t+4t-32 \\
		0 &\leq n-32 \\
		32&\leq n.
	\end{align*}
	Thus, $n \leq 31$ is not possible. Hence, it is enough to prove the conjecture for $n \geq 32$.    
	
	Define $R:= \{ u\in T^c : d(u)=4 \text{ and } d_T(u) =1\}$.
	\begin{lemma}\label{adjacent-vertex-in-R}
		Let $G$ be a graph on $n \geq 32$ vertices. Let $T$ be independent. If there exist vertices $u, v \in R$ such that $u\sim v$, then $a(G) \leq 2$.
	\end{lemma}
	\begin{proof}
		
		\textbf{Case $1$:} The vertices $u$ and $v$ are adjacent, and  have a common neighbor $w \in T$.
		Define $V_1=\{w\}, V_2=\{u,v\}, V_3=V(G) \setminus (V_1 \cup V_2)$. Note that, $\vert V_i \vert \leq n-2$ for $i=1,2,3$. 
		Then using the notations in Theorem \ref{part thm}, we have $t_1=e(V_1,V_2)=2, t_2=e(V_1,V_3)=1, t_3=e(V_2,V_3)=4$. Thus, we get $\alpha=2(n-3)$,
		$$
		\beta=(2+1).2.(n-3)+(1+4).1.2+(2+4).1.(n-3)=12(n-3)+10=12n-26,
		$$
		and
		$$
		\gamma=n(2.1+1.4+4.2)=14n.
		$$
		Define $p(x):= \alpha x^2-\beta x+\gamma$. Note that, since $n \geq 32$, we have $p(2)=4\alpha-2\beta+\gamma=8(n-3)-2(12n-26)+14n=-2n+28< 0$. Thus, by Remark \ref{poly_trick}, we have that $a(G) < 2$.
		
		\textbf{Case 2:} Any two adjacent vertices of $R$ do not have a common neighbor in $T$. Let $u$ and $v$ be in $R$ such that $u \sim u^{'}, v \sim v^{'}, u^{'} \neq v^{'}$ where $u^{'}, v^{'} \in T$.
		
		Define $X=\{u,v,u^{'},v^{'}\},\quad Q= N(X)\setminus X, \quad  Y=V(G) \setminus \{X\cup Q\}$.
		Since $d(u)=d(v)=4, d(u^{'})=d(v^{'})=3$ and $G[X]$ has exactly three edges, we must have $\vert \partial X \vert =14-2(3)=8$ and $\vert Q \vert \leq 8$. Note that since $u, v \in R$, $T \cap Q =\phi$. Now consider,
		\begin{align*}
			\vert \partial Y \vert &=e(Q,Y) \leq \sum\limits_{k \in Q}d(k) - \vert \partial X\vert \\
			&= \sum\limits_{k \in Q }d(k) -8\\
			&= \sum\limits_{k \in Q }(d(k)-4) + 4\vert Q  \vert  -8\\
			&\leq \sum\limits_{k \in T^c }(d(k)-4) + 4\vert Q  \vert  -8\\
			&=t -8 + 4\vert Q \vert-8 \quad \quad  \text{(Using Observation \ref{excess degree}).}\\ 
			&=t-16+4 \vert Q \vert.
		\end{align*}
		Since $\vert Q \vert \leq 8$, we get,\\
		\begin{align*}
			2\vert Y \vert - \vert \partial Y \vert &\geq 2(n-4-\vert Q\vert)-t+16 -4\vert Q \vert\\
			&=2n-t-6\vert Q \vert +8 \\
			&\geq 2n-(t+40).
		\end{align*}
		Note that, since $T$ is independent, using \eqref{size of T} we have that,  $t \leq \dfrac{2n-4}{3}$. Therefore, for $n \geq 32$,
		\begin{equation*}
			t+40 \leq \dfrac{2n-4}{3}+40 < 2n.
		\end{equation*}
		Hence we get, $\dfrac{\vert \partial Y \
			\vert}{\vert Y \vert }< 2$.
		By contradiction, assume that $a(G) > 2$.
		Since $\frac{\vert \partial Y \vert}{\vert Y \vert}< 2$ and $\frac{\vert \partial X \vert}{\vert X \vert}= 2$, by using Theorem \ref{Liu} for the sets $X$ and $Y$, we have $e(X,Y)^2 > 0 $, which is a contradiction. Hence, $a(G) \leq 2$.
	\end{proof}
	
	Therefore, from now on, we can assume that $R$ is either empty or forms an independent set. For $x \in T$, define $r_x= \vert  N_R(x)\vert $. Since $d(x) =3$ and $T$ is independent, we must have $r_x \leq 3$ for all $x \in T$. We next prove that if this maximum degree is attained in $R$ then the conjecture is true.
	\begin{lemma}\label{large-R-degree}
		Let $G$ be a graph on $n \geq 32$ vertices. Let $T$ and $R$ form an independent set. If there exists a vertex $x \in T$ such that $r_x=3$, then $a(G) < 2$.
	\end{lemma}
	\begin{proof}
		Let $N=n-4$. Define the vector $z\in \mathbb{R}^n$ as follows: 
		$$
		z_w =
		\begin{cases}
			3, & \text{if } w = x,\\
			1, & \text{if } w \in N(x),\\
			\frac{-6}{N}, & \text{otherwise.}
		\end{cases}
		$$
		Then $z \perp \mathbf{1}$.
		There are three edges from $x$ to $N(x)$, and since $R$ is independent, there are nine edges from $N(x)$ to the remaining $N$ vertices. Therefore,
		\begin{align*}
			z^TL(G)z&=\sum_{\{u,v\}\in E}(z_u-z_v)^2\\
			&=3(4)+9\bigg(1+\frac{6}{N}\bigg)^2 \\
			&=12 +9\bigg(1+\frac{36}{N^2}+\frac{12}{N}\bigg)\\
			&=12 +9+\frac{324}{N^2}+\frac{108}{N}\\
			&=\dfrac{21N^2+108N+324}{N^2}.
		\end{align*}
		And,
		\begin{align*}
			z^Tz=9+3+\frac{36}{N}=12+\frac{36}{N}.
		\end{align*}
		Therefore, we get,\\
		\begin{align*}
			a(G) &\leq \dfrac{z^TL(G)z}{z^Tz}\\
			&=\dfrac
			{21N^2+108N+324}{N^2(12+\frac{36}{N})}\\
			&=\dfrac
			{21N^2+108N+324}{12N(N+3)}\\
			&=\dfrac{7N^2+36N+108}{4N(N+3)}\\
			&=\dfrac{8N^2+24N+12N-N^2+108}{4N(N+3)}\\
			&=2-\dfrac{N^2-12N-108}{4N(N+3)}\\
			&=2-\dfrac{(N-18)(N+6)}{4N(N+3)}.
		\end{align*}
		Since $n \geq 32$, $N \geq 28$. Therefore $a(G) < 2$.
	\end{proof}
	
	We recall three known results which will be crucial in proving the remaining case.
	
	\begin{lemma}[{\cite{HJ12}}]\label{Schur_Comp}
		Let $Z$ be a real symmetric matrix partitioned as follows:
		\[
		Z=
		\begin{pmatrix}
			Z_{11} & Z_{12}\\
			Z_{12}^T & Z_{22}
		\end{pmatrix}.
		\]
		where $Z_{11}$ is nonsingular and let $$K=Z_{22}-Z_{12}^TZ_{11}^{-1}Z_{12}$$ be the Schur complement of $Z_{11}$ in $Z$. Then the matrix 
		$$
		\begin{pmatrix}
			Z_{11} & 0\\
			0 & K
		\end{pmatrix}.$$
		is congruent to $Z$.
	\end{lemma}
	\begin{lemma}[{\cite{HJ12}}]\label{Sylvester}
		Two real symmetric matrices are congruent if and only if they have the same inertia.  
	\end{lemma}
	\begin{lemma}\cite{HJ12}\label{Cauchy-interlacing}
		Let $A \in \mathbb{R}^{n\times n}$ be a symmetric matrix with eigenvalues
		\[
		\lambda_1 \leq \lambda_2 \leq \cdots \leq \lambda_n.
		\]
		Let $B \in \mathbb{R}^{m\times m}$ be an $m\times m$ principal submatrix of $A$, where
		$1 \leq m < n$, and let the eigenvalues of $B$ be
		\[
		\mu_1 \leq \mu_2 \leq \cdots \leq \mu_m.
		\]
		Then, for every $i=1,\ldots,m$,
		$$
		\lambda_i \leq \mu_i \leq \lambda_{i+n-m}.
		$$
		
	\end{lemma}
	\begin{theorem}\label{final-theorem}
		Let $G$ be a graph on $n \geq 32$ vertices. Suppose $T$ is independent and every degree four vertex has at most one neighbor in $T$. Also suppose $R$ is empty or $R$ is an independent set and for all $x \in T$, $r_x \leq 2$. Then $a(G) \leq 2$.
	\end{theorem}
	\begin{proof}
		Partition the matrix $L(G)-2I$ as follows:
		\begin{align*}
			L(G) - 2I=D(G)-A(G)-2I &=
			\begin{pmatrix}
				3I & 0\\
				0 & D_{T^c}
			\end{pmatrix} - 
			\begin{pmatrix}
				0 & B\\
				B^T & A(G[T^c])
			\end{pmatrix}-\begin{pmatrix}
				2I & 0\\
				0 & 2I
			\end{pmatrix}\\
			&=
			\begin{pmatrix}
				I & -B\\
				-B^T & D_{T^c}-A(G[T^c])-2I
			\end{pmatrix},
		\end{align*}
		where $D_{T^c}$ is the diagonal matrix with diagonal entries as degrees of vertices in $T^c$, $A(G[T^c])$ is the adjacency matrix of the graph induced by $T^c$, and $B$ is the $t \times (n-t) $ submatrix of $A(G)$ with rows indexed by the vertices in $T$ and the column indexed by vertices in $T^c$.
		
		Consider the Schur Complement $M$ of $I$ in $L(G)-2I$. Then $M= D_{T^c}-A(G[T^c])-2I-B^TB$. 
		Define for $u,v \in T^c$, $c_{uv}= \vert N_T(u) \cap N_T(v) \vert$. Clearly,
		\[
		(B^TB)_{u, v} = 
		\begin{cases}
			d_T(u), & u = v,\\c_{uv}, & u \neq v.
		\end{cases}
		\]
		Using this, we get that,\\
		\[
		M_{u, v} = 
		\begin{cases}
			d(u)-2-d_T(u), & u=v, \\
			-a_{uv}-c_{uv}, & u \neq v.
		\end{cases}
		\]
		where $a_{uv}$ denote the $uv$-th entry of $A(G)$. Therefore,\\
		\[
		M_{u, v} = 
		\begin{cases}
			d_{T^c}(u)-2, & u=v, \\
			-a_{uv}-c_{uv}, & u \neq v.
		\end{cases}
		\]
		Suppose by contradiction $a(G) > 2$.
		
		\textbf{Claim 1:} $\vert R \vert \leq t+1$. It is easy to see that the claim is true if $R$ is empty. Suppose $R (\neq \phi) $ forms an independent set. We shall prove that at most one $x \in T$ satisfies $r_x = 2$. Suppose on contrary, $x, y \in T$, $x \neq y$ and satisfies $r_x=r_y=2$. Let $N_R(x) =\{u,v\}$. Then, $d_{T^c}(u)=d_{T^c}(v)= 3$, $a_{uv}=0$, $c_{uv}=1$. Therefore the principal submatrix of $M$ indexed by $N_R(x)$ is:
		$$
		\begin{pmatrix}
			1 & -1\\
			-1 & 1
		\end{pmatrix}
		$$
		whose eigenvalues are $0$ and $2$.\\
		The analogous statement holds for $N_R(y)$. Moreover, since $R$ is independent and each of the vertices in $N_R(x) \cup N_R(y)$ has exactly one neighbor in $T$, we get that the principal submatrix of $M$ indexed by vertices in $N_R(x) \cup N_R(y)$ is:
		$$
		\begin{pmatrix}
			1 & -1 & 0 &0 \\
			-1 & 1 & 0 &0 \\
			0 &0 &1 &-1\\
			0 &0 &-1 &1\\
		\end{pmatrix}
		$$
		with spectrum $\{2,2,0,0\}$. By \ref{Cauchy-interlacing}, $M$ must have at least $2$ nonpositive eigenvalues. By Lemma \ref{Schur_Comp} and Lemma \ref{Sylvester}, $L(G)-2I$ must have at least two nonpositive eigenvalues. On the other hand, since $a(G) > 2$, we have that the matrix $L(G) -2I$ must have exactly one negative eigenvalue and $n-1$ positive eigenvalues. Thus, we get a contradiction. Hence, $T$ cannot have two distinct vertices that are adjacent to two vertices each in $R$. Therefore, there exists at most one vertex in $T$ satisfying $r_x=2$. Thus,
		\begin{equation*}
			\vert R \vert = \sum \limits_{x \in T}r_x\leq t+1 . 
		\end{equation*}

		Recall $H=\{u \in T^c: d(u) \geq 5\}$. Define $U=\{u \in H: d_{T^c}(u)\leq 2\}$. For $i=0,1,2$, define $U_i=\{u \in U: d_{T^c}(u)=i\}$ and $l_i=\vert U_i \vert$. It is easy to see that $U=U_0 \cup U_1 \cup U_2$ and $l:=\vert U \vert = l_0 +l_1 +l_2$.\\
		\textbf{Claim 2:} $H \neq \phi, \quad U \neq \phi$\\
		\textit{Proof:} Suppose $H = \phi$. Then all the vertices of $T^c$ have degree $4$. Since only the vertices in $R$ contribute to edges in $e(T, T^c)$, we get the following: 
		\begin{equation*}
			3t =\vert R \vert \leq t +1.
		\end{equation*}
		where the inequality follows from Claim 1. The inequality is not possible since $t \geq 8$. Thus $H \neq \phi$.\\
		Next suppose $U = \phi$. Then $\forall u \in H, d_{T^c}(u) \geq 3 \text{ and }d(u) \geq 5$. Note that every edge from $T$ to $T^c$ is incident either to a vertex in $R$ or to a vertex in $H$. Therefore we have,\\
		\begin{align*}
			3 t=\sum \limits_{u \in T^c}d_T(u) &=\sum \limits_{u \in R}d_T(u)  + \sum \limits_{u \in H}d_T(u) \\
			&= \vert R \vert + \sum \limits_{u \in H}d_T(u) \\
			&= \vert R \vert + \sum \limits_{u \in H}(d(u)-d_{T^c}(u)) \\
			&\leq \vert R \vert + \sum \limits_{u \in H}d(u)-3\vert H \vert\\
			&\leq t +1 +\sum \limits_{u \in H}(d(u)-4)+\vert H \vert\\
			&\leq t +1 +\sum \limits_{u \in H}(d(u)-4)+\sum \limits_{u \in H}(d(u)-4)\\
			&= t +1 +2\sum \limits_{u \in T^c}(d(u)-4)\\
			&=t +1+2t -16.\quad \quad \mbox{(Using Observation \ref{excess degree})}\\
		\end{align*}
		which is not possible. Thus $U \neq \phi$.
		
		Thus $U_0,U_1,U_2$ can individually be empty but not all at a time. Next we prove something stronger pertaining to the orders of $U_i$.\\
		\textbf{Claim 3:} $3l_0+2l_1+l_2 \geq 15$\\
		\textit{Proof:} As noted  in the proof of $U \neq \phi$, we have that,\\
		\begin{align}\label{CrossEdge_count}
			3t&=\vert R \vert + \sum \limits_{u \in H} d_T(u) \notag\\
			\implies \sum \limits_{u \in H} d_T(u)&=3t-\vert R \vert \notag\\ 
			\implies \sum \limits_{u \in H} d_T(u)&\geq 3t-t-1=2t-1,
		\end{align}\\
		where the last inequality is due to Claim $1$.
		Also we have,\\
		\begin{align}\label{Deg_U}
			\sum \limits_{u \in U} d_T(u) &= \sum \limits_{u \in U} (d(u)-d_{T^c}(u)) \notag \\
			&= \sum \limits_{u \in U} d(u) - \bigg(\sum \limits_{u \in U_1} d_{T^c}(u)+\sum \limits_{u \in U_2} 
			d_{T^c}(u)\bigg)\notag \\
			&= \sum \limits_{u \in U} d(u) - (l_1+2l_2) \notag \\
			&= \sum \limits_{u \in U} (d(u)-4) +4l-l_1-2l_2 .
		\end{align}
		On the other hand, since $\forall u \in H \setminus U$, $d(u) \geq 5$ and $d_{T^c}(u) \geq 3$, we get,\\
		\begin{align}\label{Deg_HminusU}
			\sum \limits_{u \in H\setminus U} d_T(u) &= \sum \limits_{u \in H \setminus U} (d(u)-d_{T^c}(u)) \notag \\
			&\leq \sum \limits_{u \in H \setminus U} d(u)-3\vert H \setminus U \vert \notag \\
			&= \sum \limits_{u \in H \setminus U} (d(u)-4) +4\vert H \setminus U \vert -3\vert H \setminus U \vert \notag \\
			&\leq \sum \limits_{u \in H \setminus U} (d(u)-4) + \sum \limits_{u \in H \setminus U} (d(u)-4) \notag \\
			&= 2 \sum \limits_{u \in H \setminus U} (d(u)-4) \notag \\
			&=2\bigg (\sum \limits_{u \in H } (d(u)-4)- \sum \limits_{u \in U} (d(u)-4)\bigg)  \notag \\
			&=2\bigg (\sum \limits_{u \in T^c } (d(u)-4)- \sum \limits_{u \in U} (d(u)-4)\bigg) \notag \\
			&=2\bigg(t-8-\sum \limits_{u \in  U} (d(u)-4)\bigg), 
		\end{align}
		where the last equality is due to observation \eqref{excess degree}.
		Using equations \eqref{CrossEdge_count}, \eqref{Deg_U} and \eqref{Deg_HminusU}, we obtain,\\
		\begin{align*}
			2t -1 &\leq \sum \limits_{u \in  U} d_T(u) +\sum \limits_{u \in H\setminus U} d_T(u) \\
			&\leq 4l-l_1-2l_2+\sum \limits_{u \in  U} (d(u)-4)+2(t -8 -  \sum \limits_{u \in  U} (d(u)-4))\\
			&=4l-l_1-2l_2+2t -16-\sum \limits_{u \in  U} (d(u)-4).\\ 
		\end{align*}
		Thus we have,\\
		\begin{equation}\label{Excess_Deg_U}
			\sum \limits_{u \in  U} (d(u)-4) \leq 4l-l_1-2l_2-15.
		\end{equation}
		Now since $U \subseteq H$, $\forall u \in U$, $d(u) \geq 5$. Therefore $\sum \limits_{u \in  U} (d(u)-4) \geq  \vert U \vert =l$. Thus, from the above inequality we get,
		\begin{align*}
			l &\leq 4l-l_1-2l_2-15\\
			\implies3l_0+2l_1+l_2&\geq15.
		\end{align*}
		
		Thus we have that $15\leq 3l_0+2l_1+l_2\leq 3(l_0+l_1+l_2)=3l$. Therefore $\vert U \vert \geq 5$. Now for distinct vertices $u,v \in U$, define $w_{uv}=a_{uv}+c_{uv}$. Note that the principal submatrix of $M$ indexed by $u,v$ is of the form:
		$$
		M[\{u,v\}]=\begin{pmatrix}
			d_{T^c}(u)-2 & -w_{uv}\\
			-w_{uv} & d_{T^c}(v)-2
		\end{pmatrix}
		$$
		Since $a(G) >2$, by Lemma \ref{Schur_Comp} and Lemma \ref{Sylvester} $M$ has exactly one negative eigenvalue and all other eigenvalues are positive. By \ref{Cauchy-interlacing}, the largest eigenvalue of $M[\{u,v\}]$ must be positive. Since $d_{T^c}(u)\leq 2, d_{T^c}(v)\leq 2$, both the diagonal entries and consequently the trace of $M[\{u,v\}]$ is nonpositive. This forces the smallest eigenvalue of $M[\{u,v\}]$ to be negative. 
		Therefore, we get,\\
		\[
		\det M[\{u,v\}]<0.
		\]
		
		Thus,
		\begin{equation}\label{Neg_Det}
			w_{uv}^{2}>(2-d_{T^c}(u))(2-d_{T^c}(v)).
		\end{equation}
		
		Since $w_{uv}$ is a nonnegative integer, equation \eqref{Neg_Det} gives the following table:
		
		\begin{equation}\label{Weight_Table}
			\begin{array}{c|ccc}
				& d_{T^c}(v)=0 & d_{T^c}(v)=1 & d_{T^c}(v)=2\\
				\hline
				d_{T^c}(u)=0 & w_{uv}\ge 3 & w_{uv}\ge 2 & w_{uv}\ge 1\\
				d_{T^c}(u)=1 &  w_{uv}\ge 2 & w_{uv}\ge 2 & w_{uv}\ge 1\\
				d_{T^c}(u)=2 & w_{uv}\ge 1  &  w_{uv}\ge 1 & w_{uv}\ge 1
			\end{array}
		\end{equation}
		From the table \eqref{Weight_Table}, we get that,
		\begin{align}\label{WeightSum_Lower}
			\sum \limits_{\{u,v\} \subseteq U}w_{uv}&= \sum \limits_{\{u,v\} \subseteq U_0}w_{uv}+\sum \limits_{u \in U_0, v\in U_1}w_{uv}+\sum \limits_{u \in U_0, v\in U_2}w_{uv}+\sum \limits_{\{u,v\} \subseteq U_1}w_{uv}+\sum \limits_{u \in U_1, v\in U_2}w_{uv}+\sum \limits_{\{u,v\} \subseteq U_2}w_{uv} \notag\\
			&\geq 3\binom{l_0}{2}+2l_0l_1+l_0l_2+2\binom{l_1}{2}+l_1l_2+\binom{l_2}{2}.
		\end{align}
		We also have that,
		\begin{align}\label{Adj_Bound}
			\sum \limits_{\{u,v\} \subseteq U}a_{uv} &=e(U)= \frac{1}{2}\sum \limits_{u \in U}d_U(u) \notag\\
			&\leq \frac{1}{2}\sum \limits_{u \in U}d_{T^c}(u) \notag \\
			&=\frac{1}{2}(l_1+2l_2).
		\end{align}
		Further since $\forall x\in T$, $d(x)=3$, we have that $0 \leq d_U(x) \leq 3$. Therefore,
		\begin{align}\label{CommonNbr_Bound}
			\sum \limits_{\{u,v\} \subseteq U}c_{uv} &=\sum \limits_{x \in T, d_U(x)=0}0  +\sum \limits_{x \in T, d_U(x)=1}0+  \sum \limits_{x \in T, d_U(x)=2}1+\sum \limits_{x \in T, d_U(x)=3}3 \notag \\
			&\leq \sum \limits_{x \in T, d_U(x)=0}d_U(x)  +\sum \limits_{x \in T, d_U(x)=1}d_U(x)+  \sum \limits_{x \in T, d_U(x)=2}d_U(x)+\sum \limits_{x \in T, d_U(x)=3}d_U(x) \notag\\
			&=\sum \limits_{x \in T}d_U(x) \notag \\
			&=\sum \limits_{u \in U}d_{T}(u) \notag \\
			&= \sum \limits_{u \in U} (d(u)-4) +4l-l_1-2l_2 \quad \quad \mbox{Using \eqref{Deg_U}}\notag \\
			&\leq 4l-l_1-2l_2-15+4l-l_1-2l_2  \quad \quad  \mbox{Using \eqref{Excess_Deg_U}} \notag \\
			&=8l-2l_1-4l_2-15.
		\end{align}
		Using \eqref{Adj_Bound} and \eqref{CommonNbr_Bound} we get that,\\
		\begin{align}\label{WeightSum_Uppr}
			\sum \limits_{\{u,v\} \subseteq U}w_{uv}&=\sum \limits_{\{u,v\} \subseteq U}(a_{uv}+c_{uv}) \notag \\
			&\leq \frac{1}{2}(l_1+2l_2) +8l-2l_1-4l_2-15 \notag \\
			&= 8l -\frac{3}
			{2}l_1-3l_2-15.
		\end{align}
		Thus using \eqref{WeightSum_Lower} and \eqref{WeightSum_Uppr} we get ,\\
		\begin{align*}
			3\binom{l_0}{2}+2l_0l_1+l_0l_2+2\binom{l_1}{2}+l_1l_2+\binom{l_2}{2}&\leq 8l -\frac{3}
			{2}l_1-3l_2-15\\
			3l_0(l_0-1)+4l_0l_1+2l_0l_2+2l_1(l_1-1)+2l_1l_2+l_2(l_2-1)&\leq 16(l_0+l_1+l_2)-3l_1-6l_2-30\\
			3l_0^2+4l_0l_1+2l_0l_2+2l_1^2+2l_1l_2+l_2^2-3l_0-2l_1-l_2&\leq 16l_0+13l_1+10l_2-30.
		\end{align*}
		Define 
		\begin{equation}\label{Contra}
			F(l_0,l_1,l_2):=3l_0^2+4l_0l_1+2l_0l_2+2l_1^2+2l_1l_2+l_2^2-19l_0-15l_1-11l_2+30 \leq 0.
		\end{equation}
		Set $\sigma=3l_0+2l_1+l_2$. By Claim $3$, we have $\sigma \geq 15$. A direct calculation shows that\\
		\begin{equation*}
			3F(l_0,l_1,l_2)=\sigma^2-19\sigma+90+2l_1^2+2l_1l_2+2l_2^2-7l_1-14l_2.
		\end{equation*}
		Since $\sigma \geq 15$, we get that $\sigma^2-19\sigma+90=(\sigma-10)(\sigma-9) \geq 30$.\\
		Further, we have, 
		\begin{align*}
			2l_1^2+2l_1l_2+2l_2^2-7l_1-14l_2&=2(l_1^2+l_1l_2-\frac{7}{2}l_1)+2(l_2^2-7l_2)\\
			&=2\left(l_1^2+l_1(l_2-\frac{7}{2})\right)+2\left(l_2-\frac{7}{2}\right)^2-\frac{49}{2}\\
			&=2\left(l_1+\frac{l_2-\frac72}{2}\right)^2
			-\frac12\left(l_2-\frac72\right)^2+2\left(l_2-\frac{7}{2}\right)^2
			-\frac{49}{2}\\
			&=2\left(l_1+\frac{l_2-\frac72}{2}\right)^2
			+\frac32\left(l_2-\frac72\right)^2
			-\frac{49}{2} \notag\\
			&\ge -\frac{49}{2}.
		\end{align*}
		Thus, we get,\\
		\[
		3F(l_0,l_1,l_2)\ge 30-\frac{49}{2}
		=\frac{11}{2}>0.
		\]
		That is, $F(l_0,l_1,l_2) > 0$. This contradicts the established inequality \eqref{Contra}. \\
		Therefore, we must have $a(G) \leq 2$.
	\end{proof}

	\section{Final remarks}\label{final-rem}
	
	In this paper, we proved that among all graphs with $n$ vertices and $2n-4$ edges, the complete bipartite graph $K_{2,n-2}$ has the largest algebraic connectivity. This settles a conjecture of Kolokolnikov \cite{Kolo15} from 2015. Note that the complete bipartite graph $K_{2,n-2}$ is not the unique maximizer in general. If $Q_3$ is the $3$-dimensional binary cube graph, then $Q_3$ has $8$ vertices, $12=2(8-2)$ edges and $a(Q_3)=a(K_{2,6})=2$. If $M_8$ is the  M\"obius ladder obtained from $C_8$ by adding edges between opposite pair of vertices, then $M_8$ has $8$ vertices, $12=2(8-2)$ edges and $a(M_8)=a(K_{2,6})=2$. If $P^{+}$ is the Petersen graph with one extra edge, then $P^{+}$ has $10$ vertices, $16=2(10-2)$ edges and $a(P^{+})=2$. 
	
	As already observed in \cite{Kolo15}, the complete bipartite graph $K_{b,n-b}$ does not have the largest algebraic connectivity among all graphs with $n$ vertices and $b(n-b)$ edges for $b\geq 8$. 
	
	We provide an example in the Appendix to show that, in general, $K_{3,n-3}$ does not maximize the algebraic connectivity among the graphs on $n$ vertices and $3(n-3)$ edges.
	\appendix
	
	\section*{Acknowledgments}
	
	This work was initiated during the visit of Sebastian M.~Cioab\u{a} to IIT Hyderabad in October 2025 for delivering a series of lectures under the GIAN program. The authors thank IIT Hyderabad and GIAN for providing this opportunity. Abhay Jayarajan acknowledges support from the DST INSPIRE Fellowship (IF220692). M.~Rajesh Kannan acknowledges financial support from ANRF-CRG, India (File No.\ CRG/2023/002747). Rahul Roy thanks the University Grants Commission (UGC), India, for financial support (NTA Ref. No. 231610209574).
	
	The authors acknowledge the use of ChatGPT for improving the exposition of the manuscript, refining some of the proofs and constructing the example. All mathematical statements have been independently verified by the authors, who are solely responsible for the correctness of the results presented in this paper.

	\bibliographystyle{amsplain}
	\bibliography{alg-conn-fixed-edge}

		
		
		

		
		
		
		

	
	\section{Appendix}
	\begin{lemma}\label{cong}\cite{MR107647}
		Let $A$ be a real symmetric matrix. Let $\Delta_k$ be the determinant of the $k$-th leading principal submatrix. Set $\Delta_0=1$. If for $i=1,\dots,n$, $\Delta_i \neq 0$ then $A$ is congruent to the diagonal matrix $diag \bigg (\dfrac{\Delta_1}{\Delta_0}, \dfrac{\Delta_2}{\Delta_1}, \dots , \dfrac{\Delta_n}{\Delta_{n-1}}\bigg )$
	\end{lemma}
	Define a graph $G$ with vertex set $V(G)= \{1, \dots, 14\}$ and the adjacency relation as per Table \ref{tab:adjacency}.
	\begin{table}[ht]
		\centering
		\begin{tabular}{c|l}
			\hline
			Vertex & Neighbors \\
			\hline
			$1$  & $5,9,10,11$ \\
			$2$  & $6,9,11,12$ \\
			$3$  & $9,12,13,14$ \\
			$4$  & $8,9,10,13$ \\
			$5$  & $1,8,12,13,14$ \\
			$6$  & $2,7,8,10,14$ \\
			$7$  & $6,9,10,13,14$ \\
			$8$  & $4,5,6,11,12$ \\
			$9$  & $1,2,3,4,7$ \\
			$10$ & $1,4,6,7,12$ \\
			$11$ & $1,2,8,13,14$ \\
			$12$ & $2,3,5,8,10$ \\
			$13$ & $3,4,5,7,11$ \\
			$14$ & $3,5,6,7,11$ \\
			\hline
		\end{tabular}
		\caption{Vertex-neighbor representation of the graph.}
		\label{tab:adjacency}
	\end{table}
	This graph has $14$ vertices and $3(14-3)=33$ edges. We prove that $a(G) > 3$.
	We give a proof using the Schur-complement method used in the proof of $b=2$ case.
	
	Set $T_1=\{v \in V(G) : d(v) =4\}=\{1,2,3,4\}$ and $\vert T_1 \vert= t_1$.
	
	Partition the matrix $L(G)-3I$ as follows:
	\begin{align*}
		L(G) - 3I=D(G)-A(G)-3I &=
		\begin{pmatrix}
			4I & 0\\
			0 & D_{T_1^c}
		\end{pmatrix} - 
		\begin{pmatrix}
			0 & B_1\\
			B_1^T & A(G[T_1^c])
		\end{pmatrix}-\begin{pmatrix}
			3I & 0\\
			0 & 3I
		\end{pmatrix}\\
		&=
		\begin{pmatrix}
			I & -B_1\\
			-B_1^T & D_{T_1^c}-A(G[T_1^c])-3I
		\end{pmatrix},
	\end{align*}
	where $D_{T_1^c}$ is the diagonal matrix with diagonal entries as degrees of vertices in $T_1^c$, $A(G[T_1^c])$ is the adjacency matrix of the graph induced by $T_1^c$, and $B_1$ is the $t_1 \times (n-t_1) $ submatrix of $A(G)$ with rows indexed by the vertices in $T_1$ and the column indexed by vertices in $T_1^c$. Since all the vertices in $T_1^c$ have degree $5$, we have that $D_{T_1^c}=5I$.
	
	Consider the Schur Complement $M_1$ of $I$ in $L(G)-3I$. Then $M_1= D_{T_1^c}-A(G[T_1^c])-3I-B_1^TB_1=5I-A(G[T_1^c])-3I-B_1^TB_1=2I-A(G[T_1^c])-B_1^TB_1$. 
	Define for $u,v \in T_1^c$, $c'_{uv}= \vert N_{T_1}(u) \cap N_{T_1}(v) \vert$. Clearly,
	$$
	(B_1^TB_1)_{u, v} = 
	\begin{cases}
		d_{T_1}(u), & u = v,\\c'_{uv}, & u \neq v.
	\end{cases}
	$$
	Using this, we get that,\\
	$$
	(M_1)_{u, v} = 
	\begin{cases}
		2-d_{T_1}(u), & u=v, \\
		-a_{uv}-c'_{uv}, & u \neq v.
	\end{cases}
	$$
	where $a_{uv}$ denote the $uv$-th entry of $A(G)$. Therefore,\\
	$$
	(M_1)_{u, v} = 
	\begin{cases}
		2-d_{T_1}(u), & u=v, \\
		-a_{uv}-c'_{uv}, & u \neq v.
	\end{cases}
	$$
	
	We can compute each of the entries of $M_1$ using above equation,
	
	$$
	M_1=
	\begin{pmatrix}
		1 & 0 & 0 & -1 & -1 & -1 & -1 & -1 & -1 & -1\\
		0 & 1 & -1 & -1 & -1 & -1 & -1 & -1 & 0 & -1\\
		0 & -1 & 2 & 0 & -1 & -1 & 0 & 0 & -1 & -1\\
		-1 & -1 & 0 & 1 & -1 & -1 & -1 & -1 & -1 & 0\\
		-1 & -1 & -1 & -1 & -2 & -2 & -2 & -2 & -2 & -1\\
		-1 & -1 & -1 & -1 & -2 & 0 & -1 & -1 & -1 & 0\\
		-1 & -1 & 0 & -1 & -2 & -1 & 0 & -1 & -1 & -1\\
		-1 & -1 & 0 & -1 & -2 & -1 & -1 & 0 & -1 & -1\\
		-1 & 0 & -1 & -1 & -2 & -1 & -1 & -1 & 0 & -1\\
		-1 & -1 & -1 & 0 & -1 & 0 & -1 & -1 & -1 & 1
	\end{pmatrix}.
	$$
	
	For $k \in \{1, \dots, 10\}$, let $\Delta_k$ denote the determinant of the leading $k \times k$ principal submatrix of $M_1$ and set $\Delta_0=1$. One can check that the following holds,
	$$(\Delta_1, \dots, \Delta_{10})= (1,1,1,-2,-9,-18,-17,-16,-7,-3)$$
	
	Using Lemma \ref{cong}
	we get that $M_1$ has $9$ positive and $1$ negative eigenvalue and consequently $L(G)-3I$ has $13$ positive and $1$ negative eigenvalue. Therefore, $a(G) > 3$.

	\definecolor{vertexred}{RGB}{255,0,0}
	\definecolor{vertexblue}{RGB}{173,216,230}
	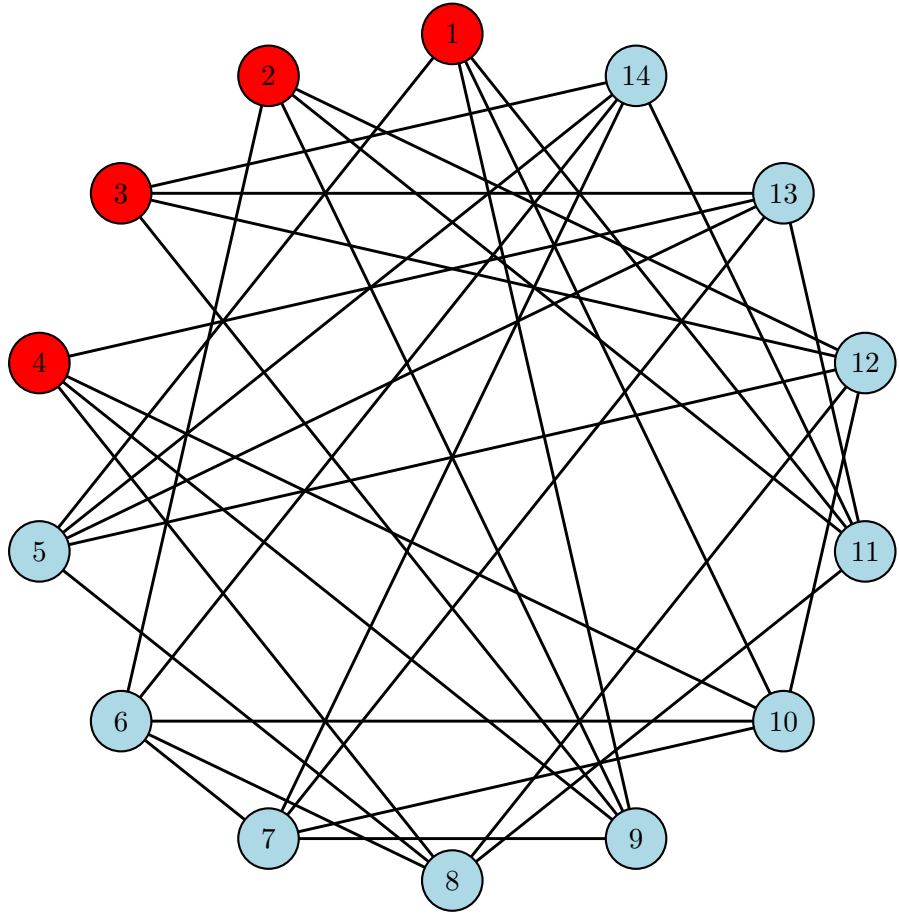
\begin{figure}[H]
		\centering
		\begin{tikzpicture}[
			edge/.style={draw=black, line width=1.1pt},
			vertex/.style={
				circle,
				draw=black,
				line width=0.8pt,
				minimum size=8mm,
				inner sep=0pt,
				font=\small
			},
			red vertex/.style={vertex, fill=vertexred},
			blue vertex/.style={vertex, fill=vertexblue}
			]
			
			
			\foreach \i in {1,...,14}{
				\pgfmathsetmacro{\vertexangle}{90+360*(\i-1)/14}
				\coordinate (p\i) at (\vertexangle:5.6cm);
			}
			
			\foreach \u/\v in {1/5,1/9,1/10,1/11,
				2/6,2/9,2/11,2/12,
				3/9,3/12,3/13,3/14,
				4/8,4/9,4/10,4/13,
				5/8,5/12,5/13,5/14,
				6/7,6/8,6/10,6/14,
				7/9,7/10,7/13,7/14,
				8/11,8/12,10/12,
				11/13,11/14}{
				\draw[edge] (p\u) -- (p\v);
			}
			
			\foreach \i in {1,...,4}{
				\node[red vertex] (v\i) at (p\i) {\i};
			}
			
			\foreach \i in {5,...,14}{
				\node[blue vertex] (v\i) at (p\i) {\i};
			}
			
		\end{tikzpicture}
		\caption{The vertices in $T_1$ are colored red}
	\end{figure}
	
	\affl{Sebastian M.~Cioab\u{a}}{cioaba@udel.edu}{ Department of Mathematical Sciences, University of Delaware, Newark, DE 19716-2553,
		USA.}

	\affl{Abhay Jayarajan}{ma23resch02001@iith.ac.in, abhayjayarajan@gmail.com}{ Department of Mathematics, Indian Institute of Technology Hyderabad, Kandi, Sangareddy 502284, India.}
	
	\affl{M. Rajesh Kannan}{rajeshkannan@math.iith.ac.in, rajeshkannan1.m@gmail.com}{Department of Mathematics, Indian Institute of Technology Hyderabad, Kandi, Sangareddy 502284, India.}
	
	\affl{Rahul Roy}{ma23resch11004@iith.ac.in, rahulroy13832@gmail.com}{ Department of Mathematics, Indian Institute of Technology Hyderabad, Kandi, Sangareddy 502284, India.}

\end{document}